\documentclass[12pt]{amsart}
\usepackage{amssymb,verbatim,amscd,amsmath,graphicx}
\usepackage{graphicx}
\usepackage{caption}
\usepackage{subcaption}
\usepackage{enumerate}
\numberwithin{equation}{section}
\newtheorem{theorem}{Theorem}[section]
\newtheorem{lemma}[theorem]{Lemma}
\newtheorem{proposition}[theorem]{Proposition}
\newtheorem{corollary}[theorem]{Corollary}

\theoremstyle{definition}

\newtheorem{def-prop}[theorem]{Definition-Proposition}

\newtheorem{remark}[theorem]{Remark}
\newtheorem{example}[theorem]{Example}

\DeclareMathOperator{\Proj}{Proj}
\DeclareMathOperator{\supp}{supp}

\DeclareMathOperator{\reg}{reg}
\DeclareMathOperator{\depth}{depth}

\DeclareMathOperator{\To}{\longrightarrow}

\DeclareMathOperator{\Ass}{Ass}

\DeclareMathOperator{\height}{ht}
\DeclareMathOperator{\lcm}{lcm}

\DeclareMathOperator{\pd}{pd}

\newcommand{\PP}{{\mathbb P}}

\newcommand{\ZZ}{{\mathbb Z}}
\newcommand{\NN}{{\mathbb N}}
\newcommand{\QQ}{{\mathbb Q}}

\def\mm{{\frak m}}

\def\pp{{\frak p}}

\def\G{{\mathcal G}}

\def\ini{\operatorname{in}}
\def\H{{\mathcal H}}

\def\D{{\Delta}}

\def\a{{\bf a}}
\def\b{{\bf b}}
\def\c{{\bf c}}

\def\1{{\bf 1}}
\def\0{{\bf 0}}

\begin{document}

%%%%%%%%%%%%%%%%%%%%%%%%%%%%%%%%%%%%%%%%%%%%%%%%%%%%%%%%%%%%%%%%%%%%%%%%%%%%%%%%%%%

\title{Depth and regularity modulo a principal ideal}

\author[Caviglia]{Giulio Caviglia}
\address{Department of Mathematics, Purdue University, 150 N. University
Street, West Lafayette, IN 47907-2067,  USA}
\email{gcavigli@math.purdue.edu}

\author[H\`a]{Huy T\`ai H\`a}
\address{Department of Mathematics \\ Tulane University \\
6823 St. Charles Ave. \\ New Orleans, LA 70118, USA}
\email{tha@tulane.edu}

\author[Herzog]{J\"urgen Herzog}
\address{Fachbereich 6, Mathematik \\ Universit\"at Duisburg-Essen \\
Campus Essen \\
45117 Essen, Germany}
\email{juergen.herzog@gmail.com}

\author[Kummini]{Manoj Kummini}
\address{Chennai Mathematical Institute \\  Siruseri, Tamilnadu 603103, India}
\email{mkummini@cmi.ac.in}

\author[Terai]{Naoki Terai}
\address{Faculty of Culture and Education, Saga University, Saga 840–8502, Japan}
\email{terai@cc.saga-u.ac.jp}

\author[Trung]{Ngo Viet Trung}
\address{Institute of Mathematics \\ Vietnam Academy of Science and Technology \\ 18 Hoang Quoc Viet, Hanoi, Vietnam}
\email{nvtrung@math.ac.vn}
\urladdr{}


\begin{abstract}
We study the relationship between depth and regularity of a homogeneous ideal $I$ and those of $(I,f)$ and $I:f$, where $f$ is a linear form or a monomial. Our results have several interesting consequences on depth and regularity of edge ideals of hypergraphs and of powers of ideals. 
\end{abstract}

\maketitle

%%%%%%%%%%%%%%%%%%%%%%%%%%%%%%%%%%%%%%%%%%%%%%%%%%%%%%%%%%%%%%%%%%

\section{Introduction}

Let $R = k[x_1, \dots, x_n]$ be a polynomial ring over a field $k$ and let $X = \Proj R/I$ be a projective scheme in $\PP^{n-1}$,
where $I$ is a homogeneous ideal. The main goal of our work is to understand the connection between algebraic invariants and properties of $X$ and its hypersurface sections. This is a classical topic that has inspired many important results in both algebraic geometry and commutative algebra.
We shall focus on invariants that govern the computational complexity, namely the depth and the Castelnuovo-Mumford regularity (or simply, regularity). More specifically, we shall examine the relationship between the depth and the regularity of $I$ and $(I,f)$, and their powers, when $f \in R$ is a linear form or when $I$ is a monomial ideal and $f$ is a monomial.

Our work is inspired by a recent paper of Dao, Huneke and Schweig \cite{DHS},
in which it was proved that if $I \subseteq R$ is a monomial ideal and $x$ is a variable appearing in the generators of $I$ then
$$\reg R/I \in \{ \reg R/(I:x) + 1, \reg R/(I,x) \},$$
and if, in addition, $I$ is a squarefree monomial ideal then
$$\depth R/I \in \{\depth R/(I:x), \depth R/(I,x)\}.$$
These results give {\em recursive formulas} for regularity and depth of
squarefree monomial ideals and played a key role in their work on edge ideals of graphs.

Our first result extends these formulas of Dao, Huneke and Schweig to a large class of homogeneous ideals.
\medskip

\noindent{\bf Theorem \ref{rmk.GLF}.}
Let $I = J + fH,$
where $J$ and $H$ are homogeneous ideals in $R$ and $f$ is a linear form that is a non-zerodivisor of $R/J$. Then
\begin{enumerate}[{\indent (i)}]
\item $\depth R/I \in \big\{ \depth R/(I:f), \depth R/(I,f)\big\}.$
\item $\reg R/I \in \big\{\reg R/(I:f)+1, \reg R/(I,f)\big\}.$
\end{enumerate}
\smallskip

Note that $I:f = J+H$ and $(I,f) = (J,f)$ under the assumptions of Theorem \ref{rmk.GLF}.
We also obtain similar formulas for the case where $f$ is any form of higher degree that is a non-zerodivisor of $R/J$.
\par

Obviously, we can choose $I$ to be any monomial ideal and $f$ to be any variable.
This particularly shows that the depth formula of Dao, Huneke and Schweig also holds for arbitrary monomial ideals. We can even extend their depth formula to the case $f$ being an arbitrary monomial as follows.  \medskip

\noindent {\bf Theorem \ref{thm.inclusion}.}
Let $I$ be a monomial ideal and let $f$ be an arbitrary monomial in $R$. Then
\begin{enumerate}[{\indent (i)}]
\item $\depth R/I \in \{\depth R/(I:f),  \depth R/(I,f)\},$
\item  $\depth R/I = \depth R/(I:f)$ if $\depth R/(I,f) \ge \depth R/(I:f)$.
\end{enumerate}
\medskip

We will give examples showing that Theorem \ref{thm.inclusion} is no longer true if $I$ is not a monomial ideal or if $f$ is not a monomial, and that there is no similar formula for $\reg R/I$ if $f$ is not a variable.
\par

We can also make the regularity formula of Dao, Huneke and Schweig more precise as follows.
\medskip

\noindent{\bf Theorem \ref{precise}.}
Let $I$ be a monomial ideal and $x$ a variable of $R$. Then
\begin{enumerate}[{\indent (i)}]
\item $\reg R/I =  \reg R/(I:x)+1$ if $\reg R/(I:x)  > R/(I,x)$,
\item  $\reg R/I \in \{\reg R/(I,x)+1, \reg R/(I,x)\}$ if $\reg R/(I:x) = \reg R/(I,x)$,
\item $\reg R/I = \reg R/(I,x)$ if $\reg R/(I:x) < \reg R/(I,x)$.
\end{enumerate}
\medskip

From Theorem \ref{precise} it follows that $\reg I$ is either
$\max\{\reg (I:x), \reg (I,x)\}$ or $\max\{\reg (I:x), \reg (I,x)\} +1$.
In particular, if $\reg (I:x) \le r$  for all variables $x$ of $R$, then $\reg I \le r+1$.
This implication is very useful because one can use induction to estimate the regularity of a monomial ideal. \par

Our approach can be used to study the behavior of the functions $\depth R/I^t$ and $\reg R/I^t$ for $t \ge 1$. 
If $I$ is the edge ideal of a hypergraph which contains a good leaf and $f$ is the monomial of the associated leaf, we show that $I^{t+1}:f = I^t$ and $\overline{I^{t+1}}:f = \overline{I^t}$ for all $t \ge 1$, where $\overline{I^t}$ denotes the integral closure of $I^t$. Using these relations we derive the following result. 
\medskip

\noindent {\bf Theorem \ref{thm.leaf}.} 
Let $I$ be the edge ideal of a hypergraph which contains a good leaf. Then
\begin{enumerate}[{\indent \rm (i)}]
\item the functions $\depth R/I^t$ and $\depth R/\overline{I^t}$ are non-increasing for all $t \ge 1$,
\item the functions $\reg R/I^t$ and $\reg R/\overline{I^t}$ are non-decreasing for all $t \ge 1$.
\end{enumerate}
\smallskip

Notice that the function $\depth R/I^t$ of the edge ideal of a hypergraph needs not be non-increasing (see \cite{HS, KSS}), and that it is still an open question (see \cite{BHH}) whether the depth function of the edge ideal of a graph is non-increasing.
\medskip

In general, the functions $\depth R/I^t$ and $\reg R/I^t$ need not be monotone; see \cite{BHH, Be, HNTT, HTT}. 
However, if we can find a common non-zerodivisor $f$ on $I^t$ for all $t \ge 1$, then $\depth R/(I,f)^t$ and $\reg R/(I,f)^t$ are monotone functions regardless of the behavior of $\depth R/I^t$ and $\reg R/I^t$.  This follows from the following result, whose proof is based on the same ideas as that of Theorem \ref{rmk.GLF}.
\medskip

\noindent {\bf Theorem \ref{thm.RE}.} 
Let $R$ be a positively graded algebra over a field, and let $I$ be a graded ideal in $R$.
Let $f \neq 0$ be a form in $R$ which is a non-zerodivisor of $R/I^t$ for all $t \le s$. Then
\begin{enumerate}[{\indent \rm (i)}]
\item $\depth R/(I,f)^s  = \min_{t \le s}   \depth R/I^t -1$,
\item $\reg R/(I,f)^s - s\deg f = \max_{t \le s} \big\{ \reg R/I^t - t\deg f \big\}$.
\end{enumerate}
\smallskip

As a consequence, if $f = x$ is a new variable then $\depth R[x]/(I,x)^s$ is a non-increasing function and $\reg R[x]/(I,x)^s$ is a non-decreasing function, while the functions $\depth R/I^s$ and $\reg R/I^s$ need not be so. \par

The relationship between depth and regularity of the ideals $I, I:x, (I,x)$ is very useful in finding combinatorial
characterizations of these invariants for edge ideals of hypergraphs. We will demonstrate this approach in two applications.

Let $I(\H)$ denote the edge ideal of a hypergraph $\H$ in a polynomial ring $R$. We call $\H$ a Cohen-Macaulay hypergraph if $R/I(\H)$ is a Cohen-Macaulay ring, i.e., $\depth R/I(\H) = \dim R/I(\H)$. Moreover, for simplicity, we set $\reg \H = \reg I(\H)$ and call it the regularity of $\H$. \par

The first application is a short proof for the following characterization of  Cohen-Macaulay very well-covered graphs, which is essentially due to Crupi, Rinaldo, and Terai \cite{CRT} (see also Constantinescu and Varbaro \cite{CV} and Mahmoudi et al \cite{Ma} for other variants of this characterization). This characterization is a generalization of a criterion for Cohen-Macaulay bipartite graphs of Herzog and Hibi \cite{HH2}.
\medskip

\noindent{\bf Theorem \ref{CM}.}
Let $\H$ be a simple graph on $2n$ vertices which has a minimal vertex cover of $n$ elements.
Then $\H$ is a Cohen-Macaulay graph if and only if $\H$ is a twin-free very well-covered graph.
\medskip

The second application is a sufficient condition for a hypergraph to have regularity $\le 3$.
Since the maximal generating degree of a homogeneous ideal is bounded above by its regularity,
the edge ideal of a hypergraph of regularity 2 is that of a graph. 
Fr\"oberg \cite{Fr} characterized graphs of regularity $\le 2$ to be those whose complements are chordal. \par

For every vertex $x$ of a hypergraph $\H$, we denote by $\H:x$ the 
hypergraph of all minimal sets of the form $F \cap (V\setminus \{x\})$, where $F$ is an edge of $\H$. 
\medskip

\noindent{\bf Theorem \ref{reg3}.}
Let $\H$ be a hypergraph such that $\H:x$ is a graph whose complement is chordal for all vertices $x$ of $\H$.
Then $\reg \H \le 3$.
\medskip

A special case of Theorem \ref{reg3} is a result of Nevo on the regularity of gap-free and claw-free graphs \cite{Ne}, which was also reproved by Dao, Huneke and Schweig in \cite{DHS}.

In publishing this paper, we hope that our results and methods
would stimulate further investigations on the behavior of depth and regularity between the ideals
$I$, $I:f$ and $(I,f)$.

%%%%%%%%%%%%%%%%%%%%%%%%%%%%%%%%%%%%%%%%%%%%%%%%%%%%%%%%%%%%%%%%%%

\section{Preliminaries} \label{sec.prel}

In this section, we recall notations and terminology used in the paper, and a number of auxiliary results.
Generally, we will follow standard texts in this research area (cf. \cite{BH, E, HH1}).

Throughout the paper, the two important invariants that we investigate are the depth and the regularity.  Though these notions can be defined in several ways, we prefer their definition by means of local cohomology modules.

Let $R$ be a positively graded algebra and let $\mm$ be its maximal homogeneous ideal.
Let $M$ be a finitely generated graded $R$-module. Let $H_\mm^i(M)$, for $i \ge 0$, denote the $i$-th local cohomology module of $M$ with respect to $\mm$. We define
\begin{align*}
\depth M & := \min\{ i ~\big|~ H^i_\mm(M) \not= 0\},\\
\reg M & := \max\{i+j ~\big|~ H^i_\mm(M)_j \not= 0\}.
\end{align*}

\begin{remark} \label{a-invariant}
Let $a_i(M) := \max \{ j|\ H^i_\mm(M)_j \not= 0\}$
with the convention that $a_i(M) := -\infty$ if $H^i_\mm(M) = 0.$
Then
$$\depth M = \min\{i|\ a_i(M) \neq -\infty\},$$
$$\reg M = \max \{a_i(M) + i ~\big|~ i \ge 0\}.$$
\end{remark}

When $R$ is a polynomial ring and $M$ is a finitely generated graded $R$-module, the depth and the regularity of $M$ are closely related to the minimal free resolution and graded Betti numbers of $M$ in the following way. Suppose that $M$ admits the following minimal free resolution:
$$0\rightarrow \bigoplus_{j \in \ZZ}R(-j)^{\beta_{p,j}(M)} \rightarrow \dots \rightarrow \bigoplus_{j \in \ZZ} R(-j)^{\beta_{0,j}(M)} \rightarrow M \rightarrow 0.$$
Let $\pd M$ denote the \emph{projective dimension} of $M$. Then
\begin{align*}
\pd M & = \max \{i ~\big|~ \beta_{i,j}(M) \not= 0 \text{ for some } j \in \ZZ\},\\
\depth M & = n - \pd M, \\
\reg M & = \max \{ j-i ~\big|~ \beta_{i,j}(M) \not= 0\}.
\end{align*}
Note that the maximal degree of the minimal generators of $M$ is bounded by $\reg M$. \par

The definition of depth and regularity by means of the local cohomology modules is best suited when working with short exact sequences. In fact, using the long derived sequence of local cohomology modules we immediately obtain the following facts, which we shall need later in the next sections.

\begin{lemma} \label{exact}
Let $0 \to L \to M \to N \to 0$ be a short exact sequence of finite graded $R$-modules. Then
\begin{enumerate}[{\indent \em (i)}]
\item $\depth M \ge \min\{\depth L, \depth N\}$,
\item $\depth M = \depth N$ if $\depth L > \depth M$,
\item $\reg L = \reg N+1$ if $\reg L > \reg M$.
\end{enumerate}
\end{lemma}

From now on, unless otherwise stated, we let $R = k[x_1, \dots, x_n]$ be a polynomial ring over a field $k$ and let $I$ be a nonzero homogeneous ideal in $R$. From the above characterizations of depth and regularity by means of the graded Betti numbers we immediately have the following relationships between $R/I$ and $I$: \begin{align*}
\depth R/I & = \depth I -1,\\
\reg R/I & = \reg I - 1.
\end{align*}
These formulas will be used without reference in this paper.

Recall that for a finitely generated graded $R$-module $M$, an element $f \in R$ is called a \emph{filter-regular element} of $M$ if $f \not\in \pp$ for all associated primes $\pp \neq \mm$ of $M$.
A sequence $f_1, \dots, f_s \in R$ of elements  is called a \emph{filter-regular sequence} of $M$ if $f_i$ is a filter-regular element of $M/(f_1, \dots, f_{i-1})M$ for all $i = 1, \dots, s$. It is easy to see that this is equivalent to the condition that
$$[(f_1,...,f_{i-1})M:f_i]_t = [(f_1,...,f_{i-1})M]_t$$
for all $i = 1,...,s$, and $t \gg 0$ (see \cite[Lemma 2.1]{Tr1}).

\begin{remark} \label{full}
If $k$ is an infinite field, we can use prime avoidance to find a linear form which is filter-regular with respect to $M$. Applying this fact successively, we may assume that $x_n,...,x_1$ is a filter-regular sequence of $M$ after a linear change of the variables.
\end{remark}

Let $\ini(I)$ denote the initial ideal of $I$ with respect to the degree reverse lexicographic order. When the $x_n,...,x_1$ is a filter-regular sequence of $R/I$, one can easily pass the investigation from $I$ to $\ini(I)$.

\begin{theorem} \label{ini}
Assume that $x_n,...,x_1$ is a filter-regular sequence of $R/I$. Then
\begin{enumerate}[{\indent \em (i)}]
\item $\depth R/I = \depth R/\ini(I),$
\item $\reg R/I = \reg R/\ini(I).$
\end{enumerate}
\end{theorem}

\begin{proof}
This follows from Remark \ref{a-invariant} and a general result \cite[Theorem 1.2]{Tr2}, which states that
$$\max\{a_i(R/I) +i|\ i \le t\} = \max\{a_i(R/\ini(I))+i|\ i \le t\}$$
for every $t \ge 0$.
\end{proof}

Our investigation on monomial ideals is based on a formula of Takayama.
Let $R = k[x_1,...,x_n]$ and $I$ a monomial ideal in $R$.
Then $R/I$ has a natural $\NN^n$-graded structure inherited from that of $R$.
Therefore, the local cohomology module $H_\mm^i(R/I)$ has a $\ZZ^n$-graded structure.
Takayama's formula \cite[Theorem 1]{Ta} describes the dimension of the $\ZZ^n$-graded component $H_\mm^i(R/I)_\a$, for $\a \in \ZZ^n$, in terms of a simplicial complexes $\D_\a(I)$.
We shall recall the construction of $\D_\a(I)$, as given in \cite{MT},
which is simpler than the original construction of \cite{Ta}.

For $\a = (a_1,...,a_n) \in \ZZ^n$, set $x^\a = x_1^{a_1} \cdots x_n^{a_n}$ and $G_\a := \{j \in [1,n] ~\big|~ a_j < 0\}$.
For every subset $F \subseteq [1,n]$, let $R_F = R[x_j^{-1}|\ j \in F]$. Define
\begin{equation*}
\D_\a(I)  = \{F \setminus G_\a|\ G_\a \subseteq F,\  x^\a \not\in IR_F\}.
\end{equation*}
We call $\D_\a(I)$ a \emph{degree complex} of $I$. \par

Takayama's formula is stated as follows:
$$\dim_kH_\mm^i(R/I)_\a =  \dim_k\widetilde H_{i-|G_\a|-1}(\D_\a(I),k).$$
The original formula in \cite[Theorem 1]{Ta} is slightly different. It contains additional conditions on $\a$ for  $H_\mm^i(R/I)_\a = 0$. However, the proof in \cite{Ta} shows that we may drop these conditions, which is more convenient for our investigation.

From Takayama's formula we immediately obtain the following characterizations of depth and regularity of monomial ideals in terms of the degree complexes.

\begin{proposition} \label{complex}
Let $I \subseteq R$ be a monomial ideal. Then
	\begin{enumerate}[\rm (i)]
		\item $\depth R/I =  \min\{|G_\a|+i ~\big|~  \a \in \ZZ^n, i \ge 0, \widetilde H_{i-1}(\D_\a(I),k) \neq 0\},$
		\item $\reg R/I =  \max\{|\a|+|G_\a|+i ~\big|~ \a \in \ZZ^n, i \ge 0, \widetilde H_{i-1}(\D_\a(I),k) \neq 0\}.$
	\end{enumerate}
\end{proposition}

%%%%%%%%%%%%%%%%%%%%%%%%%%%%%%%%%%%%

\section{Depth and regularity modulo a linear form}

The main aim of this section is to prove the following theorem, which extends the aforementioned results of Dao,
Huneke and Schweig on monomial ideals to a large class of homogeneous ideals. 
Unless otherwise specified, let $R = k[x_1, \dots, x_n]$ be a polynomial ring over a field and let $I$ be a homogeneous ideal in $R$.

\begin{theorem} \label{rmk.GLF}
Let $I = J + fH,$ where $J$ and $H$ are homogeneous ideals in $R$ and $f$ is a linear form that is a non-zerodivisor of $R/J$. Then
\begin{enumerate}[{\indent \em (i)}]
\item $\depth R/I \in \big\{ \depth R/(J+H), \depth R/J-1\big\}.$
\item $\reg R/I \in \big\{\reg R/(J+H)+1, \reg R/J\big\}.$
\end{enumerate}
\end{theorem}

\begin{proof}
Let $S = R[x]$, where $x$ is a new variable, and $Q = (J,xH)$. Then
$$R/I = S/(Q,x-f).$$

We shall show that $x-f$ is a non-zerodivisor in $S/Q$. Indeed, it is clear that  $Q = (J,x) \cap (J,H)$.
Observe that every associated prime of $(J,x)$ is of the form $(\pp,x)$, where $\pp$ is an associated prime of $J$, and $f \not\in \pp$ because $f$ is a non-zerodivisor in $R/J$. Thus, $x-f$ does not belong to any associated prime of $(J,x)$. Hence, $(J,x): (x-f) = (J,x)$. Since every associated prime of $(J,H)$ is the extension to $S$ of an associated prime of $J+H \subseteq R$, $x-f$ does not belong to any associated prime of $(J,H)$. Therefore, $(J,H):(x-f) = (J,H)$.
Hence, we get $Q: (x-f) = Q$. \par

It follows from the fact that $x-f$ is a non-zerodivisor of $S/Q$ that
$$\depth R/I = \depth S/Q -1 = \depth Q - 2 = \depth Q/xQ - 1.$$
It is easy to see that $Q/xQ = J \oplus ((J+H)/J)(-1)$.
Thus,
$$\depth Q/xQ  =  \min\big\{\depth J,\depth (J+H)/J\big\}.$$
By Lemma \ref{exact}(i), from the exact sequence
$$0 \to  J \to (J+H) \to (J+H)/J \to 0,$$
we obtain
$$\depth (J+H) \ge \min\{\depth J, \depth (J+H)/J\}.$$
Therefore,
$$\depth R/I = \depth Q/xQ - 1 \le \depth (J+H) - 1= \depth R/(J+H).$$

Now, if $\depth R/I \neq \depth R/(J+H)$ then it follows that
$$\depth R/I < \depth R/(J+H) = \depth R/I:f.$$
Applying Lemma \ref{exact}(ii) to the exact sequence
$$0 \to R/I:f \stackrel{f} \to R/I \to R/(I,f) \to 0,$$
we get $\depth R/I = \depth R/(I,f) = \depth R/(J,f) = \depth R/J - 1$ as required for Theorem \ref{rmk.GLF}(i).
\par

The proof for Theorem \ref{rmk.GLF}(ii) proceeds similarly. Firstly, we have
$$\reg R/I = \reg S/Q = \reg Q -1 = \reg Q/xQ - 1.$$
On the other hand,
$$\reg Q/xQ = \max\{\reg J, \reg (J+H)/J+1\}.$$
Therefore,
\begin{align*}
\reg R/I  & =  \max\big\{\reg J,\reg (J+H)/J+1\big\} -1\\
& = \max\big\{\reg R/J, \reg (J+H)/J\big\}.
\end{align*}
If $\reg R/I \neq \reg R/J$ then
$$\reg R/I = \reg (J+H)/J > \reg R/J.$$
By Lemma \ref{exact}(iii), from the short exact sequence
$$0 \to  (J+H)/J \to R/J \to R/(J+H) \to 0,$$
we get $\reg (J+H)/J = \reg R/(J+H) + 1$. Hence, $\reg R/I = \reg R/(J+H)+1$ as required.
\end{proof}

It is easy to see that $(I: f) = J+H$ and $(I,f) = (J,f)$ under the assumption of Theorem \ref{rmk.GLF}.
Hence, we can rewrite the formulae of Theorem \ref{rmk.GLF} as
\begin{enumerate}[{\indent \rm (i)}]
\item $\depth R/I \in \big\{ \depth R/(I:f), \depth R/(I,f)\big\},$
\item $\reg R/I \in \big\{\reg R/(I:f)+1, \reg R/(I,f)\big\}.$
\end{enumerate}

Our proof of Theorem \ref{rmk.GLF} in fact yields a slightly more general statement where $f$ is of higher degree.

\begin{corollary} \label{cor.GLFgeneral}
Let $I = J + fH,$ where $J$ and $H$ are homogeneous ideals in $R$ and $f$ is any form of degree $d$ that is a non-zerodivisor of $R/J$. Then
\begin{enumerate}[{\indent \em (i)}]
\item $\depth R/I \in \big\{ \depth R/(J+H), \depth R/J-1\big\}.$
\item $\reg R/I \in \big\{\reg R/(J+H)+d, \reg R/J+d-1, \dots , \reg R/J\big\}.$
\end{enumerate}
\end{corollary}

\begin{proof} As in the proof of Theorem \ref{rmk.GLF}, consider $S = R[x]$, where $x$ is a new variable, and $Q = (J, x^dH)$. The same argument shows that $x^d-f$ is a non-zerodivisor in $S/Q$. Thus, we may again reduce to the statements for $Q \subseteq S$. Observe that $Q = (J, x(x^{d-1}H))$. The assertion now follows by an iterated application of Theorem \ref{rmk.GLF}.
\end{proof}

The assumption of Theorem \ref{rmk.GLF} is obviously satisfied if $I$ is any monomial ideal and $f$ is any variable.
Therefore, we have the following consequence.

\begin{corollary} \label{DHS}
Let $I$ be a monomial ideal and let $x$ be a variable in $R$. Then
\begin{enumerate}[{\indent \rm (i)}]
\item $\depth R/I \in \big\{ \depth R/(I:x), \depth R/(I,x)\big\},$
\item $\reg R/I \in \big\{\reg R/(I:x)+1, \reg R/(I,x)\big\}.$
\end{enumerate}
\end{corollary}

Note that Corollary \ref{DHS}(ii) was already proved by Dao, Huneke and Schweig \cite[Lemma 2.10]{DHS}.
They also proved Corollary \ref{DHS}(i) for squarefree monomial ideals \cite[Lemma 5.1]{DHS}. Their proofs are based
 on a result of Kummini \cite{Ku} on multigraded Betti numbers of squarefree monomial ideals. \par

The following examples show that Corollary \ref{DHS} does not hold for an arbitrary linear form $f$.

\begin{example}
Let $I = (x_1x_3,x_2x_3,x_1x_4)$ and $f = x_1-x_3$ in $R = \QQ[x_1,x_2,x_3,x_4]$. Then $\depth R/I = 2$, $\depth R/(I:f) = 1$ and $\depth R/(I,f) = 0$.
\end{example}

\begin{example}
Let $I = (x_2^3, x_2x_3x_5, x_3x_8^2x_9, x_5x_7x_9^3)$ and $f = x_3-x_5$ in $R = \QQ[x_1, \dots, x_9]$. Then $\reg R/I = 6$, $\reg R/(I:f) = 7$ and $\reg R/(I,f) = 7$.
\end{example}

Under the assumption that $f$ is general enough, Theorem \ref{rmk.GLF} gives the following result.

\begin{corollary} \label{thm.GLF}
Let $f \in R$ be a linear form, which is filter-regular in $R/I$. Then
\begin{enumerate}[{\indent \rm (i)}]
\item $\depth R/I \in \big\{ \depth R/(I:f), \depth R/(I,f)\big\}.$
\item $\reg R/I \in \big\{\reg R/(I:f)+1, \reg R/(I,f)\big\}.$
\end{enumerate}
\end{corollary}

\begin{proof}  
Without restriction we may assume that $k$ is an infinite field.
Using prime avoidance and a linear change of variables, we may assume that $f = x_n$ and \par
(1) $x_n, \dots,x_1$ form a filter-regular sequence in $R/I$,\par
(2) $x_n, \dots, x_1$ form a filter-sequence in $R/(I:x_n)$,\par
(3)  $x_{n-1}, \dots, x_1$ form a filter-sequence in $R/(I,x_n)$.\par

\noindent
Let $\ini(I)$ denote the initial ideal of $I$ with respect to the reverse lexicographic order.
By Theorem \ref{ini}, we have \par
(4) $\depth R/I = \depth R/\ini(I)$ and $\reg R/I = \reg R/\ini(I),$ \par
(5) $\depth R/I:x_n = \depth R/\ini(I:x_1)$ and $\reg R/I:x_n = \reg R/\ini(I:x_n),$\par
(6) $\depth R/(I,x_n) = \depth R/\ini(I,x_n)$ and $\reg R/(I,x_n) = \reg R/\ini(I,x_n).$\par

\noindent It is also known (see \cite[Proposition 15.12]{E}, \cite[Lemma 4.3.7]{HH1}) that \par
(7) $\ini(I : x_n) = \ini(I) : x_n,$\par
(8) $\ini(I, x_n) = \ini(I) + (x_n).$\par

\noindent Since $\ini(I)$ is a monomial ideal, there are monomial ideals $J$ and $H$ such that $\ini(I) = J +x_nH$ and $x_n$ is a non-zerodivisor in $R/J$. Obviously, we have \par
(9) $\ini(I):x_n = J+H,$\par
(10) $\ini(I) + (x_n) = (J,x_n).$\par

\noindent Using (4)-(10), we can express the asserted formulas as
\begin{align*}
\depth R/\ini(I) & \in \big\{ \depth R/(J+H), \depth R/J -1\big\},\\
\reg R/\ini(I) & \in \big\{\reg R/(J+H)+1, \reg R/J\big\}.
\end{align*}
The conclusion now follows from Theorem \ref{rmk.GLF}.
\end{proof}

Corollary \ref{thm.GLF} is only interesting in the case $\depth R/I = 0$.
If $\depth R/I > 0$, then $f$ must be a regular element; hence $\depth R/I = \depth R/(I: f)$ and $\reg R/I = \reg R/(I,f)$.

%%%%%%%%%%%%%%%%%%%%%%%%%%%%%%%%%%%%%%%%%%%%%%%%%%%%%

\section{Depth and regularity colon a monomial} \label{sec.Mon}

It is natural to ask whether one can extend the results of Dao, Huneke and Schweig to the case where $x$ is replaced  by a form $f$ of higher degree. We shall see that
the same relationship between the depths of
$R/I, R/(I:f)$ and $R/(I,f)$ as in Corollary \ref{DHS}(i) holds for any monomial ideal $I$ and any monomial $f$.
We shall also give examples showing that this is no longer true if $I$ is not a monomial ideal or if $f$ is not a monomial, and that there is no similar relationship between the regularities of $R/I, R/(I:f)$ and $R/(I,f)$ as in Corollary \ref{DHS}(ii).

The investigation in this section is based on  the following relationships between $R/I$ and $R/(I:f)$,
which were proved in \cite[Corollary 1.3]{Ra} for (i) and in \cite[Lemma 4.2]{SF} for (ii) by different methods.  
 
\begin{lemma} \label{colon}
Let $I$ be a monomial ideal and let $f$ be an arbitrary monomial in $R$. Then
\begin{enumerate}[{\indent \em (i)}]
\item $\depth R/I \le \depth R/(I:f),$
\item $\reg R/I  \ge  \reg R/(I:f).$
\end{enumerate}
\end{lemma}

\begin{proof}
(i) By Proposition \ref{complex}(i), there exists $\a \in \ZZ^n$  such that $\depth R/(I:f) = |G_\a| + i$ and
$\tilde H_{i-1}(\D_{\a}(I:f),k) \neq 0$ for some $i \ge 0$. Let $f = x^\b$, $\b \in \NN^n$.
Since $G_\a := \{j|\ a_j < 0\}$, there exists $\c \in \NN^n$ such that $\supp(\c) \subseteq G_\a$ and $G_{\a+\b-\c} = G_\a$.
For every subset $F \subset [1,n]$,  $F \supseteq G_\a$, the element $x^\c$ is invertible in $R_F := k[x_i^{-1}|\ i \in F]$.
Therefore, $x^\a \in (I:f)R_F$ if and only if $x^{\a+\b} \in IR_F$ if and only if $x^{\a+\b-\c} \in IR_F$.
By the definition of the degree complex, this implies that $\D_{\a}(I:f) = \D_{\a +\b-\c}(I).$
Hence, $\tilde H_{i-1}(\D_{\a+\b-\c}(I),k) \neq 0$.
By Proposition \ref{complex} (i), it follows that
$$\depth R/I \le |G_{\a+\b-\c}| + i = |G_\a|+ i = \depth R/(I:f).$$

(ii) By Proposition \ref{complex}(ii), there exists $\a \in \ZZ^n$ such that $\reg (I:f) = |\a| + |G_\a| + i$ and
$\tilde H_{i-1}(\D_{\a}(I:f),k) \neq 0$ for some $i \ge 0$. By the maximality of $|\a| + |G_\a| + i$ we must have $a_j = -1$ for $j \in G_\a$. Let $\c \in \NN^n$ with $c_i = b_i$ for $i \in G_\a$ and $c_i = 0$ else.
Then $\supp(\c) \subseteq G_\a$ and $G_{\a+\b-\c} = G_\a$. Moreover,
$|\a + \b-\c| \ge |\a|$.
Similarly as above, we have $\D_{\a}(I:f) = \D_{\a +\b-\c}(I).$
Hence, $\tilde H_{i-1}(\D_{\a+\b-\c}(I),k) \neq 0$.
By Proposition \ref{complex}(ii), this implies that
$$\reg R/I \ge |\a+\b-\c| + |G_\a| + i \ge |\a| + |G_\a| + i = \reg R/(I:f).$$
\end{proof}

One might expect that $\reg R/I \ge \reg R/(I:f)+ \deg f$.
As shown by the following example, this bound does not hold,
even when $f$ is a variable appearing in the generators of $I$.

\begin{example} \label{example1}
Let $R = \QQ[x_1, \dots, x_9]$ and consider
$$I = (x_2x_3, x_2x_4, x_3x_5, x_4x_5, x_2x_5x_7, x_4x_8, x_3x_9, x_4x_9, x_7x_9, x_8x_9).$$
Then, for $f = x_8$, we have $\reg R/I = 2 = \reg R/(I:f)$. Thus, $\reg R/I \not\ge \reg R/(I:f) + 1.$
\end{example}

Our next result shows that the conclusion of Corollary \ref{DHS}(i) also holds for an arbitrary monomial $f$.

\begin{theorem} \label{thm.inclusion}
Let $I$ be a monomial ideal and let $f$ be an arbitrary monomial in $R$. Then
\begin{enumerate}[{\indent \em (i)}]
\item $\depth R/I \in \{\depth R/(I:f),  \depth R/(I,f)\}.$
\item  $\depth R/I = \depth R/(I:f)$ if $\depth R/(I,f) \ge \depth R/(I:f)$.
\end{enumerate}
\end{theorem}

\begin{proof} Assume that $\depth R/I \neq \depth R/(I:f)$. Then $\depth R/(I :f) > \depth R/I$ by Lemma \ref{colon}(i).
Consider the short exact sequence
$$0 \To R/(I:f)  \stackrel{f} \To R/I \To R/(I,f) \To 0.$$
By Lemma \ref{exact}(ii), this implies that $\depth R/I = \depth R/(I,f)$.  \par
If $\depth R/(I,f) \ge \depth R/(I:f)$, then $\depth R/I \ge \depth R/(I:f)$ by Lemma \ref{exact}(i).
By Lemma \ref{colon}(i), this implies that $\depth R/I = \depth R/(I:f).$
\end{proof}

Theorem \ref{thm.inclusion}(i) does not necessarily hold if $I$ is not a monomial ideal or if $f$ is not a monomial.

\begin{example} \label{example.nomonomial}
Let $R = \QQ[x,y,z,u,v]$ and consider
$$I = (x^2y^3-z^2u^3, y^2z^3-x^2zuv, x^3u^2-y^2v^3) \text{ and } f = xyzuv.$$
Then $\depth R/I = 2$, $\depth R/(I:f) = 1$ and $\depth R/(I,f) = 0$.
\end{example}

\begin{example} \label{example.nof}
Let $R = \QQ[x,y,z,u,v]$ and consider
$$I = (x^3y, y^2z^5, z^2u^4v) \text{ and } f = xy-zv.$$
Then $\depth R/I = 2$, $\depth R/(I:f) = 1$ and $\depth R/(I,f) = 0$.
\end{example}

From the short exact sequence
$0 \To R/(I:f)  \stackrel{f} \To R/I \To R/(I,f) \To 0,$
we might expect that
$$\reg R/I \in \{\reg R/(I:f)+\deg f,  \reg R/(I,f)\},$$
which would generalize the mentioned result of Dao, Huneke and Schweig \cite[Lemma 2.10]{DHS} to the case where $\deg f > 1$. However, the following example shows that this is not possible, even when $I$ is a monomial ideal and $f$ is a monomial.

\begin{example} \label{example.noregularity}
Let $R = \QQ[x,y,z,u,v]$ and consider
$$I = (xy^2, yz^2, zu^3, uv^2x, v^2xz) \text{ and } f = x^3y.$$
Then $\reg R/I = 6, \reg R/(I,f) = 7$ and $\reg R/(I:f) + \deg f = 8$. In particular,
$\reg R/I \not\in \{\reg R/(I:f) + \deg f, \reg R/(I,f)\}.$
\end{example}

On the other hand, the formula for $\reg R/I$ of Dao, Huneke and Schweig can be made more precise as follows.

\begin{theorem} \label{precise}
Let $I$ be a monomial ideal and $x$ a variable of $R$. Then
\begin{enumerate}[{\indent \em (i)}]
\item $\reg R/I =  \reg R/(I:x)+1$ if $\reg R/(I:x)  > R/(I,x)$,
\item  $\reg R/I \in \{\reg R/(I,x)+1, \reg R/(I,x)\}$ if $\reg R/(I:x) = \reg R/(I,x)$,
\item $\reg R/I = \reg R/(I,x)$ if $\reg R/(I:x) < \reg R/(I,x)$.
\end{enumerate}
\end{theorem}

\begin{proof}
We only need to prove (i) and (iii). \par
If $\reg R/(I:x)  > \reg R/(I,x)$, then $\reg R/I > \reg R/(I,x)$ by Lemma \ref{colon}(ii).
Hence, $\reg R/I = \reg R/(I:x)+1$. \par
If $\reg R/(I:x) < \reg R/(I,x)$, then $\reg R/I > \reg R/(I:x)$ because $\reg R/I \ge \reg R/(I,x)$ by the proof of \cite[Lemma 2.10]{DHS}. Therefore, $\reg R/I = \reg R/(I,x)$.
\end{proof}

From Theorem \ref{precise} we immediately obtain the following estimate, which shows that $\reg I$ is either
$\max\{\reg (I:x), \reg (I,x)\}$ or $\max\{\reg (I:x), \reg (I,x)\} +1$.

\begin{corollary} \label{bounds}
Let $I$ be a monomial ideal and $x$ a variable of $R$. Then
$$\max\{\reg (I:x), \reg (I,x)\}  \le \reg I \le \max\{\reg (I:x)+1, \reg (I,x)\}.$$
\end{corollary}

In particular, we obtain the following inductive estimate for the regularity of monomial ideals.

\begin{proposition} \label{local}
Let $I$ be a monomial ideal. 
If $\reg (I:x) \le r$ for all variables $x$ of $R$, then $\reg I \le r+1$.
\end{proposition}

\begin{proof}
The case $n = 1$ is trivial. If $n > 1$, we have $(I,x):y = (I:y,x)$ for all variables $x,y$ of $R$.
Therefore, $\reg ((I,x): y) = \reg (I:y,x)  \le \reg (I:y) \le r$
by the lower bound of Corollary \ref{bounds}.
By induction, we have $\reg (I,x) \le r+1$.
Therefore, the upper bound of Corollary \ref{bounds} implies $\reg I \le r+1$.
\end{proof}

Using Proposition \ref{local}, we can easily recover the following result.

\begin{corollary} \label{HT} \cite[Theorem 3.1]{HT}
Let $I$ be a monomial ideal. Let $\lcm(I)$ denote the least common multiple of the minimal monomial generators of $I$.
Then
$$\reg I \le \deg \lcm(I)- \height I+1.$$
\end{corollary}

\begin{proof}
The case $\deg \lcm(I) = 1$ is trivial.
Without restriction we may assume that every variable $x$ appears in the minimal monomial generators of $I$.
Then $\deg \lcm(I:x) \le \deg \lcm(I) - 1$ and $\height (I:x) \ge \height I$.
By induction, we have
$$\reg (I:x) \le \deg \lcm(I:x) - \height (I:x) +1\le \deg \lcm(I) - \height I.$$
Therefore, the conclusion follows from Proposition \ref{local}.
\end{proof}

%%%%%%%%%%%%%%%%%%%%%%%%%%%%%%%%%%%%

\section{Depth and regularity functions} \label{sec.nonincreasing}

In this section, we investigate the depth and the regularity of powers of an ideal. Our approach allows us to study the non-increasing and non-decreasing properties of the depth and the regularity functions for particular classes of ideals.

Let $I$ be a homogeneous ideal in a polynomial ring $R$.
It is usually difficult to compute $\depth R/I^t$ and $\reg R/I^t$ for all $t \ge 1$.
As a consequence, the behavior of $\depth R/I^t$ and $\reg R/I^t$, as functions in $t$, often remains mysterious.
We only know that for $t$ sufficiently large, $\depth R/I^t$ and $\reg R/I^t$ are constant or a linear function, respectively
\cite{Br,CHT,Ko}.

Random examples show that $\depth R/I^t$ (respectively, $\reg R/I^t$) tends to be a non-increasing (respectively, non-decreasing) function, though that is not the case in general \cite{BHH, Be, HNTT,HTT}.
Bandari, Herzog and Hibi  \cite{BHH} asked whether $\depth R/I^t$ is a non-increasing function for squarefree monomial ideals.
A counter-example was recently found by combinatorial methods in \cite{KSS} (see also \cite{HS}).
However, it is still an open question whether $\depth R/I^t$ is a non-increasing function for the edge ideal of an arbitrary graph. 

Recall that a {\em hypergraph} is a collection of sets, called edges, which are subsets of a vertex set.  
Given a hypergraph $\H$ on the vertices $x_1,...,x_n$, the \emph{edge ideal} of $\H$ is defined to be
$$I(\H) := \big\langle \prod_{x \in F}x |\ F\ \text{is an edge of}\ \H \big\rangle.$$
We will use Lemma \ref{colon}(i) to present a large class of edge ideals of hypergraphs for which $\depth R/I^t$ 
(respectively, $\reg R/I^t$) is a non-increasing (respectively, non-decreasing) function. \par

An edge $F$ of a hypergraph $\H$ is called a \emph{good leaf} if the intersections of $F$ with other edges of $\H$ form a chain with respect to containment. This notion is originally introduced for simplicial complexes \cite{HHTZ}. For the study of edge ideals, one may always assume that there are no containments among the edges of the hypergraph, which can be considered as the facets of  a simplicial complex.
The leaves of a graph are always good. Good leaves do exist in the hypergraph of the facets of a simplicial forest \cite[Corollary 3.4]{HHTZ}, which may be defined as a hyperforest \cite[Theorem 3.2]{HHTZ}.

\begin{theorem} \label{thm.leaf}
Let $I$ be the edge ideal of a hypergraph which contains a good leaf. Then
\begin{enumerate}[{\indent \em (i)}]
\item the functions $\depth R/I^t$ and $\depth R/\overline{I^t}$ are non-increasing for all $t \ge 1$,
\item the functions $\reg R/I^t$ and $\reg R/\overline{I^t}$ are non-decreasing for all $t \ge 1$.
\end{enumerate}
\end{theorem}

\begin{proof}
	Let $F$ be a good leaf and $f$ the monomial associated with $F$. Let $g$ be an arbitrary monomial of 
	$I^{t+1} :f$, $t >0$. Then $fg = hf_1 \cdots f_{t+1}$, where $h$ is a monomial and $f_1,...,f_{t+1}$ are monomials associated with some edges $F_1,...,F_{t+1}$ of the hypergraph.
	By the definition of a good leaf, we may assume that $F_i \cap F \subseteq F_1 \cap F$ for all $i = 1,...,t+1$.
	Then the variables of $F \setminus F_1$ do not appear in $f_1,...,f_{t+1}$.
	So they appear in $h$. From this it follows that $hf_1$ is divisible by $f$. Therefore, $g$ is divisible by $f_2 \cdots f_{t+1}$. That means $g \in I^t$. Since $I^t \subseteq I^{t+1}:f$,  we have shown that $I^{t+1}: f = I^t$. 
	Hence, $\depth R/I^t \ge \depth R/I^{t+1}$ and $\reg R/I^t \le \reg R/I^{t+1}$ by Lemma \ref{colon}.
	
	Let $g \in \overline{I^{t+1}}: f$. Then $(gf)^m \in I^{m(t+1)}$ for some $m > 0$.
	Similarly as above, we can show that $g^m \in I^{mt}$. Hence $g \in \overline{I^t}$.
	From this it follows that $\overline{I^{t+1}}: f = \overline{I^t}$.
	Since $\overline{I^t}$ is also a monomial ideal, $\depth R/\overline{I^t} \ge \depth R/\overline{I^{t+1}}$ and $\reg R/\overline{I^t} \le \reg R/\overline{I^{t+1}}$ by Lemma \ref{colon}.
\end{proof}

In the following we give formulas for $\depth R/(I,f)^s$ and $\reg R/(I,f)^s$, when $f$ is a non-zerodivisor of $R/I^t$ for all $t \le s$. The proof for these formulas follows the same line of arguments as that of Theorem \ref{rmk.GLF}.

\begin{theorem} \label{thm.RE}
Let $R$ be a positively graded algebra over a field, and $I$ be a graded ideal in $R$.
Let $f$ be a form in $R$ which is a non-zerodivisor of $R/I^t$ for all $t \le s$. Then
\begin{enumerate}[{\indent \em (i)}]
\item $\depth R/(I,f)^s  = \min_{t \le s}   \depth R/I^t -1,$
\item $\reg R/(I,f)^s  - s\deg f= \max_{t \le s} \big\{ \reg R/I^t -t\deg f \big\}$.
\end{enumerate}
\end{theorem}

\begin{proof}
Let $S = R[x]$ with $\deg x = \deg f = d$. If we consider $(I,x)$ as a graded module over $R$, we have a decomposition
$$(I,x)^s = I^s \oplus I^{s-1}x \oplus \cdots \oplus Ix^{s-1} \oplus Rx^s\oplus Rx^{s+1}\cdots$$
From this it follows that
$$S/(I,x)^s = (R/I^s) \oplus (R/I^{s-1})(-d) \oplus \cdots \oplus (R/I)(-(s-1)d).$$
Thus, we have
\begin{align*}
\depth S/(I,x)^s & = \min_{t \le s}  \depth R/I^t,\\
\reg S/(I,x)^s - sd & = \max_{t \le s} \big\{ \reg R/I^t -td\big\}.
\end{align*}

Since $R/(I,f)^s  \simeq R[x]/((I,x)^s,x-f),$ it remains to show that $x-f$ is a non-zerodivisor of $R[x]/(I,x)^s$. Let $g \in (I,x)^{s-1}$ such that $(x-f)g \in (I,x)^s$. Without restriction we may assume that $g = c_0 + c_1x + \cdots c_{s-1}x^{s-1}$, $c_i \in R$. Then $c_0f \in I^s$. Hence, $c_0 \in I^s:f = I^s$. From this it follows that
$(c_1x + + \cdots c_{s-1}x^{s-1})(x-f) \in (I,x)^s$. This implies $c_1f \in I^{s-1}$. Hence, $c_1 \in I^{s-1}:f = I^{s-1}$.
Proceed in this fashion, we eventually get $c_i \in I^{s-i}$ for $i = 0,...,s-1$. Therefore, $g \in (I,x)^s$, as required.
\end{proof}

Theorem \ref{thm.RE} does not hold if $f \neq 0$ is not a non-zerodivisor of $R/I^t$ for all $t \le s$, as illustrated in the following example.

\begin{example} \label{example.noRE}
Let $R = k[x,y,z,u,v]$ and consider
$$I = (x^5, x^4y, xy^4, y^5, x^2y^2(xz^6-yu^6), x^3y^3) \text{ and } f = z-u.$$
Then $f$ is a non-zerodivisor of $R/I$, but $f$ is a zerodivisor $R/I^2$.

We now have
$\depth R/(I,f)^2 = 1, \depth R/I = 2, \text{ and } \depth R/I^2 = 1.$
Thus,
$$\depth R/(I,f)^2 \not= \min_{1 \le t \le 2} \depth R/I^t - 1.$$

We also have $\reg R/(I,f)^2 = 15, \reg R/I = 10, \text{ and } \reg R/I^2 = 20.$
Therefore,
$$\reg R/(I,f)^2 \not= \max_{1\le t \le 2} \{\reg R/I^t + (s-t)\deg f\}.$$
\end{example}

From Theorem \ref{thm.RE}(i) we immediately obtain the following consequence.

\begin{corollary} \label{increasing}
Let $R$ be a positively graded algebra over a field and $I$ a graded ideal of $R$.
Let $f$ be a form in $R$ which is a non-zerodivisor of $R/I^t$ for all $t \ge 0$.
Then
\begin{enumerate}[{\indent \em (i)}]
\item $\depth R/(I,f)^t$ is a non-increasing function,
\item $\reg R/(I,f)^t - t\deg f$ is a non-decreasing function.
\end{enumerate}
\end{corollary}

As an example, we may let $f$ be a new variable $x$, then $\depth R[x]/(I,x)^t$ is a non-increasing function and
$\reg R[x]/(I,x)^t - t\deg f$ is a non-decreasing function. This displays an interesting phenomenon because the functions $\depth R/I^t$ and $\reg R/I^t$ needs not be so (see \cite[Theorem 0.1]{BHH}, \cite[Theorem 6.7]{HNTT} and \cite[Example 2.8]{HTT}).

\begin{remark}
Theorem \ref{thm.RE}(i) and Corollary \ref{increasing}(i) also hold for ideals in a \emph{local} ring as it can be seen from the proof of Theorem \ref{thm.RE}(i).
\end{remark}

In the following we denote by $\Ass I$ the set of associated primes of $I$ in a Noetherian ring $R$.
Brodmann \cite{Br} showed that $\Ass I^t = \Ass I^{t+1}$ for $t \gg 0$.
In general, we may have $\Ass I^t \not\subseteq \Ass I^{t+1}$ (see e.g. \cite{BHH, HNTT, HS}).
It has been of great interest to know when $\Ass I^t \subseteq \Ass I^{t+1}$ for all $t \ge 1$
(see e.g. \cite{FHV, HQ, MMV}).
\par

The following consequence of Theorem \ref{thm.RE}(i) shows that if there is a non-zerodivisor $f$ modulo $I^t$ for all $t \ge 1$, then $\Ass (I,f)^t \subseteq \Ass (I,f)^{t+1}$ for all $t \ge 1$.

\begin{corollary} \label{ass}
Let $R$ be a Noetherian ring and $I$ an ideal of $R$.
Let $f \in R$ be an element which is a non-zerodivisor of $R/I^t$ for all $t \le s$.
Then $\Ass(I,f)^t \subseteq \Ass (I,f)^{t+1}$ for all $t < s$.
\end{corollary}

\begin{proof}
Let $P \in \Ass (I,f)^t$, $t <s$. Then $\depth (R/(I,f)^t)_P = 0$.
By the local version of Theorem \ref{thm.RE}(i), we have $\depth (R/I^j)_P = 1$ for some $j \le t$, and hence
$$\depth (R/(I,f)^{t+1})_P = \min_{j \le t+1} \depth (R/I^j)_P - 1 = 0.$$
This implies that $P \in \Ass (I,f)^{t+1}$, as required.
\end{proof}

%%%%%%%%%%%%%%%%%%%%%%%%%%%%%

\section{Depth and regularity of edge ideals}

The aim of this section is to show that our approach can be used to study edge ideals of hypergraphs.
Throughout this section we assume that all hypergraphs are simple and have no isolated vertices. \par

We will first investigate very well-covered graphs.
Recall that a graph is {\em well-covered} if all minimal vertex covers have the same size. A graph is {\em very well-covered} if all minimal vertex covers contain exactly half of the vertices. In these definitions we can replace minimal vertex covers by maximal independent sets because they are complements of one another. \par

It is well known that every associated prime of the edge ideal of a graph is generated by the elements of a minimal vertex cover of the graph, and that this gives a correspondence between these two notions. Therefore, the edge ideal a graph is (height) unmixed if and only if the graph is well-covered. An unmixed edge ideal of height $n$ in a polynomial ring of $2n$ variables must come from a very well-covered graph. The interest in edge ideals of very well-covered graphs originates from bipartite graphs because an unmixed bipartite graph must be very well-covered.
\par

A graph $\H$ is called {\em Cohen-Macaulay} if $R/I(\H)$ is a Cohen-Macaulay ring, i.e., $\depth R/I(\H) = \dim R/I(\H)$.
It is well known that the edge ideal of a Cohen-Macaulay graph is always unmixed.
Herzog and Hibi \cite[Theorem 3.4]{HH2} first gave a characterization of Cohen-Macaulay bipartite graphs. This result
was extended to very well-covered graphs by Crupi, Rinaldo and Terai \cite{CRT} (see also Constantinescu and Varbaro \cite{CV} and Mahmoudi et al \cite{Ma}).
We will give a short proof of their results. Our proof is based on the following effective characterizations of very-well covered graphs, due to Favaron \cite{Fa}.

\begin{lemma} \label{covered} \cite[Theorem 1.2]{Fa}
A graph $\H$ is very well-covered if and only if it has a perfect matching $M$ with the following two properties:
\begin{enumerate}[{\indent \rm (i)}]
\item no edge of $M$ belongs to a triangle in $\H$,
\item if an edge of $M$ is the central edge of a path of length 3 in $\H$, then the two vertices at the ends of the path must be adjacent.
\end{enumerate}
Moreover, if $\H$ is very well covered, then every perfect matching in $\H$ satisfies these properties.
\end{lemma}

Recall that two vertices $x, y$ of a graph are {\em twin} if they have the same (open) neighborhood.
A twin-free graph is sometimes referred to as \emph{irreducible} by authors in graph theory.

\begin{lemma} \label{twin} \cite[Theorem 2.5]{Fa}
A very well-covered graph $\H$ is twin-free if and only if it has a perfect matching $M$ such that there is no cycle $C_4$ in $\H$ with two edges of $M$.
\end{lemma}

Our result on the Cohen-Macaulayness of very well-covered graphs is formulated in a different way than those of Constantinescu and Varbaro \cite[Theorem 2.3]{CV}, Crupi, Rinaldo and Terai \cite[Theorem 3.4]{CRT} and Mahmoudi et al \cite[Theorem 3.2]{Ma}.

\begin{theorem} \label{CM}
Let $\H$ be a simple graph on $2n$ vertices which has a minimal vertex cover of $n$ elements.
Then $\H$ is a Cohen-Macaulay graph if and only if $\H$ is a twin-free very well-covered graph $\H$.
\end{theorem}

\begin{proof}
Without restriction, we may assume that $\H$ is a very well-covered graph.

If $\H$ is not twin-free, then $\H$ has a cycle $C_4$ which contains two edges of a perfect matching of $\H$. It follows
from Lemma \ref{covered}(i) that the opposite vertices of this $C_4$ are not adjacent.
Therefore, each maximal independent set in $\H$ contains exactly one pair of opposite vertices of this $C_4$. This implies that the independence complex of $\H$ is not strongly connected and, thus, is not Cohen-Macaulay; see \cite[Proposition 11.7]{Bj}. Hence, $\H$ is not Cohen-Macaulay.
\par

Conversely, suppose that $\H$ is twin-free. Let $M = \{\{x_1,y_1\},...,\{x_n,y_n\}\}$ be a perfect matching in $\H$. Using Lemmas \ref{covered} and \ref{twin}, it is easy to check that for any proper subset $K \subseteq [1,n]$, the induced subgraph of $\H$ on the vertices $\{x_i,y_i\}_{i \in K}$ is also a twin-free very well-covered graph. By induction, we may assume that these induced subgraphs are Cohen-Macaulay.
\par

Consider a vertex $x$ in $\H$ and let $y$ be the vertex matched with $x$ in $M$. Let $N(x)$ denote the set of vertices in $\H$ adjacent to $x$. Observe that if there exists $z \in N(x)$ with $z \not= y$ then by Lemma \ref{covered}(ii) we have $N(y) \subseteq N(z)$. Furthermore, this inclusion is strict since $y$ and $z$ are not twin. Thus, by taking $x$ such that $|N(y)|$ is largest possible, we may choose $x$ so that $N(x) = \{y\}$. Without loss of generality, suppose $N(x_n) = \{y_n\}$.

Let $R = k[x_1, \dots, x_n, y_1, \dots, y_n]$ and $I = I(\H)$.
Let $\H'$ be the induced subgraph of $\H$ on $\{x_1, \dots, x_{n-1}, y_1, \dots, y_{n-1}\}$ and $R' = k[x_1, \dots, x_{n-1}, y_1, \dots, y_{n-1}]$. Then $(I,y_n) = (I(\H'),y_n)$.
Since $R'/I(\H')$ is a Cohen-Macaulay ring with $\dim R'/I(\H') = n-1$, we have
$$\depth R/(I,y_n) = \depth S/I(\H') + 1 = n.$$

Without loss of generality we may assume that $\{y_1, \dots, y_n\}$ is a maximal independent set of $\H$.
Suppose $N(y_n) = \{x_{i_1}, \dots, x_{i_s}\}$. Let $\G$ denote the induced subgraph of $\H$ on $Z :=  \{x_i, y_i|\ i \neq i_1,...,i_s\}$ and $S = k[Z]$.
By Lemma \ref{covered}(ii), the vertices $y_{i_1}, \dots, y_{i_s}$ are not adjacent to any vertex $z \in Z$.
This implies that $(I:y_n) = I(\G) + (x_{i_1}, \dots, x_{i_s})$.
Since $S/I(\G)$ is a Cohen-Macaulay ring with $\dim S/I(\G) = n-s$, we have
$$\depth R/(I:y_n) = \depth S/I(\G) + s = n.$$
By Theorem \ref{thm.inclusion}, we have $\depth R/I \in \{\depth R/(I:y_n), \depth R/(I,y_n)\}$.
Therefore, $\depth R/I = n = \dim R/I$.
\end{proof}

We shall now discuss an application of Proposition \ref{local} on the regularity of edge ideals.
For a hypergraph $\H$ we set $\reg \H = \reg I(\H)$ and call it the regularity of $\H$. \par

If $\reg \H \le 2$, the degree of the minimal monomial generators of $I(\H)$ is bounded by $2$. 
Hence, $I(\H)$ is actually the edge ideal of a graph. In particular, $\H$ is a graph if there are no containments among the edges. If $\H$ is a graph, we know that $\reg \H \le 2$ if and only if the complement of $\H$ is chordal \cite{Fr}. 
Recall that a graph is called \emph{chordal} if it has no induced cycle $C_m$ with $m \ge 4$. 
In the following we give a necessary condition for $\reg \H \le 3$.
\par

For a vertex $x$ of a hypergraph $\H$, we denote by $\H:x$ the hypergraph of all minimal sets of the form 
$F \cap (V\setminus \{x\})$, where $F$ is an edge of $\H$. 

\begin{theorem} \label{reg3}
Let $\H$ be a hypergraph such that $\H:x$ is a graph whose complement is chordal for all vertices $x$ of $\H$.
Then $\reg \H \le 3$.
\end{theorem}

\begin{proof}
Let $I = I(\H)$. Then $I:x = I(\H:x)$ for all vertices $x$ of $\H$.
By the above result of Fr\"oberg, the assumption on $\H:x$ implies that $\reg (I:x) \le 2$.
Therefore, $\reg I \le 3$ by Proposition \ref{local}.
\end{proof}

The above necessary condition for $\reg \H \le 3$ is not so far from being a sufficient condition because due to Corollary \ref{bounds}, $\reg \H \le 3$ if and only if $\reg (\H-x) \le 3$ or $\reg (\H:x) \le 2$ for some vertex $x$. Here, $\H-x$ denotes the hypergraph of the edges of $\H$ not containing $x$. If $\reg (\H:x) \le 2$, then $\H:x$ is a graph (there are no containments among the edges of $\H:x$) whose complement is chordal by \cite{Fr}. \par

As a consequence, we obtain the following result of Nevo \cite[Theorem 5.1]{Ne}, which was also reproved by Dao, Huneke and Schweig \cite[Theorem 3.4]{DHS}. Recall that a {\em claw} of a graph is an induced complete bipartite graph $K_{1,3}$ and a {\em gap} is an induced graph whose complement is a cycle $C_4$.

\begin{corollary}
Let $\H$ be a claw-free and gap-free graph. Then $\reg \H \le 3$.
\end{corollary}

\begin{proof}
By Theorem \ref{reg3}, we only need to show that the complement of $\H:x$ is a chordal graph for all $x \in V$.
Assume that there is a vertex $x$ such that the complement of $\H:x$ is not chordal.
Then the complement of $\H:x$ contains an induced cycle $C_m$, $m \ge 4$.
Let $G$ denote the induced subgraph of $\H$ whose complement is this $C_m$.
Let $y$ be an adjacent vertex of $x$. Then $y$ is not a vertex of $G$.
If $y$ were not adjacent to an edge of $G$, the induced graph on $x,y$ and the vertices of this edge were a gap.
Therefore, $y$ must be adjacent to every edge of $G$. Let $z$ be a vertex of $G$ adjacent to $y$. Let $u,v$ be
the two vertices of $G$ non-adjacent to $z$. Then $\{u,v\}$ is an edge of $G$. Hence, $y$ is adjacent to one of them, say $u$. Since $y$ is adjacent to $x,z,u$ and since $x,z,u$ are not adjacent to each other, the induced subgraph on $x,y,z,u$ is a claw, a contradiction.
\end{proof}

\noindent {\bf Acknowledgement}.
Giulio Caviglia is partially supported by the Simons Foundation (Grant \#209661).
Huy Tai Ha is partially supported by the Simons Foundation (Grant \#279786).
Ngo Viet Trung is partially supported by Vietnam National Foundation for Science and Technology Development (Grant \#101.04-2017.19) and the project VAST.HTQT.NHAT.1/16-18.
The main part of this work was done during research stays of the authors at the American Institute of Mathematics in the SQuaRE program ``Ordinary powers and symbolic powers" during the period 2012-2014. 
The authors are grateful to S. A. Seyed Fakhari for pointing out a mistake of Theorem \ref{thm.leaf} in the first version of this paper and to an anonymous referee for mentioning that Lemma \ref{colon}(i) and (ii) were already proved in \cite{Ra} and \cite{SF}, respectively.

%%%%%%%%%%%%%%%%%%%%%%%%%%%%%%%%%%%%%%%%%%%%%%%%%%%%%%%%%%%%%%%%%%


\begin{thebibliography}{1}

\bibitem{BHH} S. Bandari, J. Herzog, and T. Hibi,
{\it Monomial ideals whose depth function has any given number of strict local maxima},
Ark. Mat. {\bf 52} (2014), 11--19.
\bibitem{Be} D. Berlekamp,
{\it Regularity defect stabilization of powers of an ideal},
Math. Res. Lett. {\bf 19} (2012), 109--119.
\bibitem{Bj} A. Bj\"orner, {\it Topological methods}, in: Handbook of combinatorics, Vol. 2, 1819-1872, Elsevier, 1995.
\bibitem{Br} M. Brodmann,
{\it The asymptotic nature of the analytic spread},
Math. Proc. Cambridge Philos. Soc. {\bf 86} (1979),  35--39.
\bibitem{CV} A. Constantinescu and M. Varbaro,
{\it On the $h$-vectors of Cohen-Macaulay flag complexes},
Math. Scand. {\bf 112} (2013), 86--111.
\bibitem{CRT} M. Crupi, G. Rinaldo and N. Terai,
{\it Cohen-Macaulay edge ideals whose height is half of the number of vertices},
Nagoya Math. J. {\bf 201} (2011), 116--130.
\bibitem{CHT} S.D. Cutkosky, J. Herzog and N.V. Trung, {\it Asymptotic behaviour of the Castelnuovo-Mumford regularity}, Compositio Math. {\bf 118} (1999), no. 3, 243--261.
\bibitem{DHS} H.L. Dao, C. Huneke and J. Schweig, {\it Bounds on the regularity and projective dimension of ideals associated to graphs}, J. Algebraic Combin. {\bf 38} (2013), no. 1, 37--55.
\bibitem{E} D. Eisenbud, Commutative Algebra: with a View Toward Algebraic Geometry, Springer, 1995.
\bibitem{Fa} O. Favaron, {\it Very-well covered graphs},
Discrete Math. {\bf 42}, no. 2--3, (1982), 177--187.
\bibitem{FHV} C. A. Francisco, H. T. H\`a, A. Van Tuyl,
{\it Colorings of hypergraphs, perfect graphs, and associated primes of powers of monomial ideals},
J. Algebra {\bf 331} (2011), 224--242.
\bibitem{Fr} R. Fr\"oberg, {\it On Stanley-Reisner rings}, in: Topics in algebra, Banarch Center Publications, {\bf 26} (2) (1990) 57--70.
\bibitem{HNTT} H.T. H\`a, H.D. Nguyen, N.V. Trung and T.N. Trung, 
{\it Symbolic powers of sums of ideals}, 
Preprint, arXiv:1702.01766.
\bibitem{HTT} H. T. H\`a, N. V. Trung and T. N. Trung, 
{\it Depth and regularity of powers of sums of ideals}, 
Math. Z. {\bf 282} (2016), no. 3--4, 819--838.
\bibitem{HS} H.T. H\`a and M. Sun, {\it Squarefree monomial Ideals that fail the persistence property and non-increasing depth}, Acta Math. Vietnam. {\bf 40} (2015), 125--137.
\bibitem{HH} J. Herzog and T. Hibi, 
{\it The depth of powers of an ideal}, 
J. Algebra {\bf 291} (2005), 534--550.
\bibitem{HH2} J. Herzog and T. Hibi,
{\it Distributive lattices, bipartite graphs and Alexander duality},
J. Algebraic Combin. {\bf 22}, no. 3 (2005), 289--302.
\bibitem{HH1} J. Herzog and T. Hibi, Monomial Ideals, Springer, 2011.
\bibitem{HHTZ} J. Herzog and T. Hibi, N. V. Trung and X. Zheng, 
{\it Standard vertex cover algebras, cycles and leaves}, 
Trans. Amer. Math. Soc. {\bf 360}, no. 12 (2008), 6231--6249.
\bibitem{HQ} J. Herzog and A.A. Qureshi,
{\it Persistence and stability properties of powers of ideals},
J. Pure Appl. Algebra {\bf 219}, no 3 (2015), 530--542.
\bibitem{HT} L.T. Hoa and N.V. Trung,
{\it On the Castelnuovo-Mumford regularity and the arithmetic degree of monomial ideals},
Math. Z. {\bf 229}, no. 3 (1998), 519--537.
\bibitem{KSS} T. Kaiser, M. Stehl\'ik, and R. Skrekovski, 
{\it Replication in critical graphs and the persistence of monomial ideals}, 
J. Combin. Theory Ser. A {\bf 123} (2014), 239--251.
\bibitem{Ko} V. Kodiyalam, 
{\it Asymptotic behaviour of Castelnuovo-Mumford regularity},  
Proc. Amer. Math. Soc. {\bf 128}, no. 2, (1999), 407--411.
\bibitem{Ku} M. Kummini, 
Homological invariants of monomial and binomial ideals, 
Thesis, University of Kansas, 2008.
\bibitem{Ku2} M. Kummini, 
{\it Regularity, depth and arithmetic rank of bipartite edge ideals}, 
J. Algebr. Comb. {\bf 30} (2009), 429--445.
\bibitem{Ma} M. Mahmoudi, A. Mousivand, M. Crupi, G. Rinaldo, N. Terai, and S. Yassemi,
{\it Vertex decomposability and regularity of very well-covered graphs},
 J. Pure Appl. Algebra {\bf 215} (2011), 2473--2480.
\bibitem{MMV} J. Martinez-Bernal, S. Morey, R. H. Villarreal,
{\it Associated primes of powers of edge ideals},
Collect. Math. {\bf 63}, no 3 (2012), 361--374.
\bibitem{MT} N.C. Minh and N.V. Trung,
{\it Cohen-Macaulayness of monomial ideals and symbolic powers of Stanley-Reisner ideals},
Adv. Math. {\bf 226} (2011), 1285--1306.
\bibitem{Ne} E. Nevo, {\it Regularity of edge ideals of $C_4$-free graphs via the topology of the lcm-lattice},
J. Combin. Theory Ser. A {\bf 118} (2011), 491--501.
\bibitem{Ra} A. Rauf, 
{\it Depth and Stanley depth of multigraded modules}, 
Commun. Algebra {\bf 38} (2010), 773--784.
\bibitem{SF} S. A. Seyed Fakhari, 
{\it Symbolic powers of cover ideal of very well-covered and bipartite graphs}, 
Proc. Amer. Math. Soc. {\bf 146} (2018), 97--110.
\bibitem{Ta} Y. Takayama, {\it Combinatorial characterizations of generalized Cohen-Macaulay monomial ideals}, Bull. Math. Soc. Sci. Math. Roumanie (N.S.) {\bf 48} (2005), 327--344.
\bibitem{Tr1} N. V. Trung, {\it Reduction exponent and degree bounds for the defining equations of a graded ring}, Proc. Amer. Math. Soc. {\bf 101} (1987), 229--236.
\bibitem{Tr2} N. V. Trung, {\it Gr\"obner bases, local cohomology and reduction number}, Proc. Amer. Math. Soc. {\bf 129}  (2001), 9--18.
\end{thebibliography}
\end{document}